\documentclass[11pt,reqno]{amsart}
\usepackage{enumerate}
\usepackage{amssymb,commath,amscd}
\usepackage{txfonts}
\usepackage{amsthm}
\usepackage{mathrsfs}
\usepackage{thmtools}
\usepackage{amsmath}
\usepackage[colorlinks,linkcolor=black,anchorcolor=blue,citecolor=green]{hyperref}
\usepackage[all]{xy}

\declaretheorem[numberwithin=section]{theorem}
\declaretheorem[numberwithin=section]{lemma}
\declaretheorem[numberwithin=section]{corollary}

\declaretheorem[numberwithin=section]{problem}
\declaretheorem[numberwithin=section]{conjecture}

\declaretheorem[sibling=theorem, style=remark]{remark}
\declaretheorem[sibling=theorem, style=remark]{example}
\declaretheorem[sibling=theorem, style=remark]{definition}
\numberwithin{equation}{section}
\allowdisplaybreaks

\newcommand{\op}[1]{\operatorname{#1}}
\begin{document}
	% \title[short text for running head]{full title}
	\title[Rigidity for Einstein manifolds] {Rigidity for Einstein manifolds under bounded covering geometry}
	
	%    Only \author and \address are required; other information is
	%    optional.  Remove any unused author tags.
	
	%    author one information
	% \author[short version for running head]{name for top of paper}
	\author{Cuifang Si}
	\address{School of Mathematical Sciences, Capital Normal Universiy, Beijing China}
	%\curraddr{}
	\email{2210501004@cnu.edu.cn}
	%\thanks{}
	\author{Shicheng Xu}
	\address{School of Mathematical Sciences, Capital Normal Universiy, Beijing China}
	\address{Academy for Multidisciplinary Studies, Capital Normal University, Beijing  China}
	%\curraddr{}
	\email{shichengxu@gmail.com}

	%\subjclass is required.
	
	
	\date{May 31, 2025}
	
	\maketitle
	
	\begin{center}
		\emph{Dedicated to Professor Xiaochun Rong for his 70th birthday}
	\end{center}
	
		\begin{abstract}
		In this note we prove three rigidity results for Einstein manifolds with bounded covering geometry. 
		(1) An almost flat manifold $(M,g)$ must be flat if it is Einstein, i.e. $\op{Ric}_g=\lambda g$ for some real number $\lambda$.
		(2) A compact Einstein manifolds with a non-vanishing and almost maximal volume entropy is hyperbolic. 
		(3) A compact Einstein manifold admitting a uniform local rewinding almost maximal volume is isometric to a space form. 
	\end{abstract}

	\section{Introduction}\label{section-0}
	Several years ago
	Xiaochun Rong proposed a program that aims to study the geometry and topology of manifolds with lower bounded Ricci curvature and local bounded covering geometry \cite{HKRX2020}, see the survey paper \cite{HRW2020}. Following the program, we prove three rigidity results for Einstein manifolds under local bounded covering geometry.
	
	The first one is that an almost flat manifold must be flat if it is Einstein.
	
	Let us recall that Gromov's theorem (\cite{Gromov1978}, \cite{Ruh1982}) on almost flat manifolds says that for any positive integer $n>0$, there are $\epsilon(n), C(n)>0$ such that for any compact $n$-manifold $(M,g)$, if the rescaling invariant $$\op{diam}(M,g)^2\cdot \max |\op{Sec}_g|< \epsilon(n),$$ then $M$ is diffeomorphic to a infra-nilmanifold $N/\Gamma$, where $N$ is a simply connected nilpotent group, and $\Gamma$ is a subgroup of affine group $N\rtimes\op{Aut}(N)$ of $N$ such that $[\Gamma:\Gamma\cap N]\le C(n)$.
	
	There are many compact Ricci-flat manifolds including complex $K3$ surfaces and $G2$ manifolds. Though it generally fails, Gromov's theorem, however, still holds at the level of fundamental group of manifolds under lower bounded Ricci curvature. That is, there is $\epsilon(n), C(n)>0$ such that for any compact $n$-manifold $(M,g)$ of almost non-negative Ricci curvature, i.e., $$\op{diam}(M,g)^2\cdot \op{Ric}_g>-\epsilon(n),$$ its fundamental group $\pi_1(M)$ is virtually $C(n)$-nilpotent, i.e. $\pi_1(M)$ contains a nilpotent subgroup $N$ with index $[\pi_1(M):N]\le C(n)$. It was originally conjectured by Gromov \cite{GLP1981}, and proved by Kapovitch-Wilking \cite{KW}, which now is called the generalized Margulis lemma. In fact, the polycyclic rank of $\pi_1(M)$ is well-defined and is no more than $n$, which is that of any finite-indexed nilpotent subgroup $N$ (cf. \cite[\S 2.4]{NaberZhang2016}), i.e., the number of $\mathbb Z$ appears in a polycyclic series as cyclic factors.
	
	Recently a criterion for almost flat manifold theorem under lower bounded Ricci curvature was proved by Huang-Kong-Rong-Xu \cite{HKRX2020}. Combining with Naber-Zhang \cite{NaberZhang2016}, we have the following result.
	
	\begin{theorem}[Almost flat theorem under lower bounded Ricci curvature, \cite{HKRX2020}, \cite{NaberZhang2016}]\label{thm-almostflat}
		There is $\epsilon(n)>0, v(n)>0$ such that for any $n$-manifold
		$(M,g)$ of Ricci curvature $\ge -(n-1)$ and $\op{diam}(M,g)< \epsilon(n)$, the followings are equivalent:
		\begin{enumerate}
			\item $M$ is diffeomorphic to a infra-nil manifold;
			\item the polycyclic rank of $\pi_1(M)$ is equal to $n$;
			\item $(M,g)$ satisfies $(1,v(n))$-bounded covering geometry, i.e. 
			$\op{vol} B_1(\tilde x)\ge v(n)>0$, where $\tilde x$ is a preimage point of $x$ in the Riemannian universal cover $\widetilde{M}=\widetilde{B_1(x)}$.
		\end{enumerate}
	\end{theorem}
	
	In the above theorem, it is well-known by the structure of nilpotent Lie group that the polycyclic rank of $N\cap \pi_1(M)$ is equal to $n$. Hence, (1) trivially implies (2). It was proved in \cite[Theorem A]{HKRX2020} that (3) implies (1), and by \cite[Proposition 5.9]{NaberZhang2016}, (2) implies (3).
	
	Theorem \ref{thm-almostflat} can be viewed as a natural extension of Colding's maximal Betti number theorem \cite{Colding1997}, i.e, under the condition of Theorem \ref{thm-almostflat}, $M$ is diffeomorphic to a $n$-torus if and only if its first Betti number equals $n$. 
	
	In this note, we prove that if in addition $(M,g)$ is Einstein in Theorem \ref{thm-almostflat}, then $(M,g)$ must be flat.
	
	\begin{theorem}[Rigidity for almost nonnegative Ricci curvature]\label{thm-rigidity-almostflat}
		There is $\epsilon(n)>0, v(n)>0$ such that for any Einstein $n$-manifold
		$(M,g)$ with $\op{Ric}_g=\lambda g$ and $\op{diam}(M,g)^2\cdot \lambda>-\epsilon(n)$, the followings are equivalent:
		\begin{enumerate}
			\item $M$ is diffeomorphic to a infra-nilmanifold;
			\item the polycyclic rank of $\pi_1(M)$ is equal to $n$;
			\item $(M,g)$ satisfies $(1,v(n))$-bounded covering geometry after a rescaling on the metric $g$ such that $\op{diam}(M,g)=\epsilon(n)$;
			\item $M$ is flat.
		\end{enumerate}
	\end{theorem}
	
	Note that by Milnor \cite{Milnor1976}, any left invariant metric on a non-abelian nilpotent group is not Einstein. At the same time, there are many almost flat metrics which are not left invariant. It follows form Theorem \ref{thm-rigidity-almostflat} all ``sufficiently'' almost flat manifolds must be flat if it is Einstein. 
	
	\begin{corollary}\label{cor-almost-flat-Einstein-mfd}
		There is $ \epsilon(n)>0 $ such that for any almost flat $n$-manifold $ (M,g) $, if 
		$$\op{diam}(M,g)^2\cdot \max |\op{Sec}_g|< \epsilon(n),$$
		and $(M,g)$ is Einstein, then $(M,g)$ is flat.
	\end{corollary}
	
	If in addition, the first Betti number of $ M $ in Corollary \ref{cor-almost-flat-Einstein-mfd} is equal to $ n $, then we have the following rigidity.
	\begin{corollary}[Rigidity for Colding’s torus stablity]
		There is $ \epsilon(n)>0$, $ v(n)>0 $ such that for any Einstein $ n $-manifold $ (M,g) $ with $ \op{Ric}_g=\lambda g $ and $ \op{diam}(M,g)^2\cdot\lambda>-\epsilon(n) $, the followings are equivalent:
		\begin{enumerate}
			\item $ M $ is a flat torus;
			\item the first Betti number of $ M $ is equal to $ n $.
		\end{enumerate}
	\end{corollary}
	
	Note that any Einstein metric on $ T^4 $ is flat; see \cite{Besse1987}. However, it is unknown for $ T^n $, $ n\ge 5 $.
	
	The second rigidity in this paper is for Einstein manifolds with almost maximal non-vanishing volume entropy. 
	
	Let us recall that for a compact Riemannian $n$-manifold $(M,g)$, its volume entropy is defined to be
	$$h(M,g)=\lim_{R\to +\infty}\frac{\ln \op{vol}B_R(\tilde x)}{R},$$
	where $\tilde x$ is an arbitrary fixed point in the Riemannian universal cover $(\tilde M,\tilde g)$. By Manning \cite{Manning1979}, $h(M,g)$ is well-defined and does not depend on the choice of $\tilde x$. If $(M,g)$ has Ricci curvature $\op{Ric}_g\ge -(n-1)$, then by Bishop's volume comparison $h(M,g)\le (n-1)$.
	By Ledrappier-Wang \cite{Ledrappier-Wang2010} (cf. Liu \cite{Liu2011}), equality holds if and only if $(M^n,g)$ is isometric to a hyperbolic manifold $\mathbb{H}^n/\Gamma$. 
	
	The following quantitative rigidity was proved by Chen-Rong-Xu \cite{CRX2019} for manifolds with negatively lower bounded Ricci curvature and almost maximal volume entropy. We will use $\varkappa(\epsilon|n,D)$ to denote a positive function depends on $\epsilon$, $n$, $D$ such that it converges to zero as $\epsilon\to 0$ with other parameters fixed.
	
	\begin{theorem}[Quantitative rigidity for almost maximal volume entropy \cite{CRX2019}, cf. \cite{CX2024}]\label{thm-almost-rigid}
		There exists $\epsilon(n, D)>0$ such that for $0<\epsilon<\epsilon(n, D)$, if a compact Riemannian $n$-manifold $(M,g)$ with $\op{Ric}_g\ge -(n-1)$ satisfies 
		$$h(M,g)\geq n-1-\epsilon, \quad \op{diam}(M,g)\leq D,$$
		then $(M,g)$ is diffeomorphic and $\varkappa(\epsilon | n, D)$-Gromov-Hausdorff close to a hyperbolic $n$-manifold $\mathbb{H}^n/\Gamma$. 
		
		Conversely, if $(M,g)$ is $\epsilon$-Gromov-Hausdorff close to a $\mathbb{H}^n/\Gamma$, then $$|h(M,g)-(n-1)|\le \varkappa(\epsilon|n,D).$$
	\end{theorem}	
	
	Note that it was proved by Gromov-Thurston \cite{Gromov-Thurston87} that, for any integer $n\ge 4$ and any positive constant $\delta$, there exists a compact Riemannian $n$-manifold $(M, g)$ such 
	that the sectional curvature of $(M,g)$ satisfies $$-1-\delta \le  \op{Sec}_g \le -1,$$ 
	while the manifold $M$ does not admit any hyperbolic metric.
	More recently, Schroeder-Shah \cite{Schroeder-Shah2018} proved that for every $n\ge 3$ and $\epsilon>0$, there exists $v>0$ and a compact Riemannian $n$-manifold $M$ such that 
	$$-1\le \op{Sec}_g\le 0, \quad \op{vol}(M,g)\le v, \quad h(M,g)\ge (n-1)-\epsilon,$$ and $M$ does not admit any hyperbolic metric.	
	
	The next theorem is a gap phenomena of maximal volume entropy for Einstein manifolds,  which follows from Theorem \ref{thm-almost-rigid} and the fact that any hyperbolic metric is locally rigid \cite[\S 12.73]{Besse1987}, \cite{Koiso1978}.
	\begin{theorem}[Gap for non-vanishing and maximal volume entropy]\label{thm-gap}
		For any integer $n\ge 3$ and $D>0$, there is $\epsilon(n,D)>0$ such that if a compact Einstein Riemannian $n$-manifold $(M,g)$ with $\op{Ric}_g\ge -(n-1)$ satisfies 
		$$h(M,g)\ge n-1-\epsilon(n,D), \quad \op{diam}(M,g)\leq D,$$
		then $h(M,g)=n-1$ and $(M,g)$ is hyperbolic.
	\end{theorem}

	It was proved in \cite{CRX2019} that for any Riemannian $n$-manifold $(M,g)$ satisfies $\op{Ric}_g\ge -(n-1)$ and $\op{diam}(M,g)\leq D$, the following conditions are equivalent as $\epsilon\to 0$:
	\begin{enumerate}
		\item $h(M,g)\ge n-1-\epsilon$,
		\item $\frac{\op{vol}B_1(\tilde x)}{\op{vol}B_1^{-1}}\ge 1-\epsilon$ for any $\tilde x$ in the Riemannian universal cover $(\tilde M,\tilde g)$, where $B_1^{-1}$ is a unit ball in the simply connected hyperbolic space $\mathbb H^n$.
	\end{enumerate} 
	
	Hence manifolds with non-vanishing and almost maximal volume entropy admits a global bounded covering geometry. It was conjectured that if the universal cover in (2) is replaced by the local universal cover $\widetilde{B_1(x)}$, then Theorem \ref{thm-almost-rigid} still holds.  
	
	More generally, the following conjecture was raised by Chen-Rong-Xu \cite{CRX2019}.
	\begin{conjecture}[Quantitative maximal local rewinding volume
		rigidity \cite{CRX2019}]\label{conj-maxlocalrewindingvol}
		Given integer $n>0$ and $H=\pm 1$ or $0$, there exists a constant
		$\epsilon(n,\rho)>0$ such that for any $0<\epsilon<\epsilon(n,\rho)$, if a compact Riemannian $n$-manifold $(M,g)$ satisfies
		$$\op{Ric}_g\ge (n-1)H, \quad \frac{\op{vol}B_\rho (x^*)}{\op{vol} B_\rho^H}\ge 1-\epsilon, \text{ for any $x\in M$}, $$
		where $x^*$ is a preimage point of $x$ in the universal cover $\widetilde{B_\rho(x)}$, and $B_\rho^H$ is a $\rho$-ball in the simply connected space form of constant curvature $H$,
		then $M$ is diffeomorphic and $\varkappa(\epsilon|n,\rho)$-close to a space form of constant
		curvature $H$, provided that $\op{diam}(M)\le D$ (and, thus, $\epsilon(n, \rho, D)$) when $H\neq 1$.	
	\end{conjecture}
	
	Conjecture \ref{conj-maxlocalrewindingvol} has been verified for the case of Einstein manifolds in \cite[Theorem E]{CRX2019} and the case of manifolds with two-sided bounded Ricci curvature \cite{CRX2017}. If in addition, the global Riemannian universal cover $(\tilde M,\tilde x)$ satisfies the non-collapsing condition, $ \op{vol}(B_1(\tilde x))\ge v>0$, then it was proved in \cite{CRX2019} that Conjecture \ref{conj-maxlocalrewindingvol} holds for $\varkappa(\epsilon|v,n,\rho,D)$.
	
	The last main result is the rigidity for Einstein manifolds in Conjecture \ref{conj-maxlocalrewindingvol}, which improves \cite[Theorem E]{CRX2019}.
	
	\begin{theorem}[Rigidity of maximal local rewinding volume]\label{thm-rigid-localrewindingvol}
		Let $(M,g)$ be a Riemannian $n$-manifold satisfying
		$$\op{Ric}_g\ge (n-1)H-\epsilon,\quad \op{diam}(M,g)\le D, \quad \frac{\op{vol}B_\rho (x^*)}{\op{vol} B_\rho^H}\ge 1-\epsilon, \text{ for any $x\in M$,}$$
		where $-1\le H\le 1$.
		If $(M,g)$ is Einstein and $0< \epsilon\le \epsilon(n,\rho,D)$, then $(M,g)$ is isometric to a space form of constant curvature $H$. 
	\end{theorem}
	
	Note that, since the Ricci curvature lower bound $H$ in Conjecture \ref{conj-maxlocalrewindingvol} is normalized to be $\pm 1$ and $0$, almost flat manifolds are already excluded by this assumption. In Theorem \ref{thm-rigid-localrewindingvol}, however, there is no such curvature normalization, so that almost flat manifolds are included in the conditions of Theorem \ref{thm-rigid-localrewindingvol}. Hence its proof is quite different than the earlier proof of \cite[Theorem E]{CRX2019} at this point, where Theorem \ref{thm-rigidity-almostflat} plays an essential rule. 
	
	Recall that it was known by Perelman \cite{Perelman1994}, Colding \cite{Colding1996}, \cite{Colding1997} and Cheeger-Colding \cite{CC1997I} that any Riemannian $n$-manifold with positive Ricci curvature $ \op{Ric}\ge n-1 $ and almost maximal volume is diffeomorphic to a round sphere $\mathbb S^n$. A special case of Theorem \ref{thm-rigid-localrewindingvol} is the following rigidity for spheres.
	\begin{corollary}[{\cite[Theorem C]{Honda-Mondello}}]\label{cor-sphere}
		There is $\epsilon(n)>0$ such that for any
		Riemannian $n$-manifold $(M,g)$ with $\op{Ric}_g\ge (n-1)$ and $\op{vol}(M,g)\ge (1-\epsilon(n))\op{vol}(\mathbb S^n)$, if $(M,g)$ is Einstein, then $M$ is isometric to the round sphere $\mathbb S^n$.
	\end{corollary}
	
	In fact, it was proved by Honda-Mondello \cite[Theorem C]{Honda-Mondello} that the Corollary \ref{cor-sphere} holds for singular Einstein stratified space without the codimension $ 2 $-stratum.
	So are all main Theorems \ref{thm-rigidity-almostflat}, \ref{thm-gap}, \ref{thm-rigid-localrewindingvol}; see Theorem \ref{Thm-stratified-rigidity} in Section \ref{startified-space}.
	
	After the paper is finished, we learned that another proof of Corollary \ref{cor-sphere} was given by Dong-Chen \cite{Dong-Chen} recently.
	
	\textbf{Acknowledgements:} This work is partially supported by NSFC (Grant No. 11821101 and No. 12271372). S. X. is grateful to Xiaoyang Chen for raising the question how an Einstein almost flat manifold can be, and Shouhei Honda for pointing out that Theorem \ref{thm-rigid-localrewindingvol} holds for singular Einstein stratified space without codimension $ 2 $ stratum, and Nan Li for helpful discussion on Example \ref{ex-iterated-metric}.
	%Thomas Schick for raising Problem \ref{problem-2}, 
	
	\section{Preliminaries}
	\subsection{Equivariant Gromov-Hausdorff convergence}
	
	The references of this part are \cite{FY1992}, \cite{Rong2010}.
	
	Let $ (X,d_X) $ and $ (Y,d_Y) $ be two compact length metric space. A map $ f: X\to Y $ is called an $ \epsilon $-Gromov-Hausdorff approximation from $ X $ to $ Y $, if $ B_{\epsilon}(f(X))=Y $ and $ |d(x_1,x_2)-d(f(x_1),f(x_2))|<\epsilon $.
	Roughly speaking, the Gromov-Hausdorff distance $ d_{GH}(X,Y) $ can be viewed as the infimum of $ \epsilon $ such that there exists an $ \epsilon $-Gromov-Hausdorff approximation from $ X $ to $ Y $ and an $ \epsilon $-Gromov-Hausdorff approximation from $ Y $ to $ X $.
	We say that $ (X_i,d_i) $ converges to $ (X,d) $ in Gromov-Hausdorff sense, denoted by $ (X_i,d_i)\overset{GH}{\longrightarrow} (X,d) $, if the Gromov-Hausdorff distance satisfies $ d_{GH}((X_i,d_i),(X,d)) \to 0, i\to \infty $.
	
	Assume $ \Gamma $ and $ G $ are closed isometric groups  on X and Y respectively. A triple of maps $ f: X\to Y,\varphi: \Gamma\to G,\psi: G\to \Gamma $ is called $ \epsilon $-equivariant Gromov-Hausdorff approximation from $ (X,\Gamma) $ to $ (Y,G) $, if for $ x\in X $, $ y\in Y $, $ t\in \Gamma $, $ s\in G $ such that 
	$$ f \text{ is a }\delta\text{-GHA},
	\quad d( \varphi(t)f(x), f(tx) )<\epsilon,
	\quad d( f(\psi(s)x), sf(x) )<\epsilon .$$
	The equivariant Gromov-Hausdorff distance $ d_{GH}((X,\Gamma),(Y,G)) $ can be viewed as the infimum of $ \epsilon $ such that there exists an $ \epsilon $-equivariant Gromov-Hausdorff approximation from $ X $ to $ Y $ and an $ \epsilon $-equivariant Gromov-Hausdorff approximation from $ Y $ to $ X $. We say that $ (X_i,G_i) $ converges to $ (X,G) $ in equivariant Gromov-Hausdorff sense, denoted by $ (X_i,G_i)\overset{GH}{\longrightarrow} (X,G) $, if $ d_{GH}((X_i,G_i),(X,G))\to 0, i\to \infty $.
	
	The followings are some properties of equivariant Gromov-Hausdorff convergence applied later.	
	\begin{lemma}\label{group-GH_converge}
		Let $ (X_i,d_i) $ and $ (X,d) $ be compact length metric space.
		If $ (X_i,d_i)\overset{GH}{\longrightarrow} (X,d) $, and $ G_i $ is the closed group of isometries on $ X_i $, then $ (X_i,G_i)\overset{GH}{\longrightarrow} (X,G) $ after passing to a subsequence, where $ G $ is a closed group of isometries on $ X $.
	\end{lemma}
	
	\begin{lemma}\label{quotient-GH-converge}
		Let $ (X_i,d_i) $ and $ (X,d) $ be compact length metric space, and $ G_i $ is the closed group of isometries on $ X_i $.
		If $ (X_i,G_i)\overset{GH}{\longrightarrow} (X,G) $, then $ X_i/G_i\overset{GH}{\longrightarrow} X/G $.
	\end{lemma}
	
	%For length metric spaces whose diameter doesn't have a uniform upper bounded,
	The above notion and properties of equivariant Gromov-Hausdorff convergence can be extended to a pointed version for complete metric spaces without a uniform diameter bound. We say $ (X_i,x_i,d_i) $ pointed Gromov-Hausdorff converges to $ (X,x,d) $, if for all $ R>0 $, the closed ball $ B_R(x_i) $ and $ B_R(x) $ with restricted metric from $ (X_i,d_i) $ and $ (X,d) $ respectively satisfies $ (B_R(x_i),d_i)\overset{GH}{\longrightarrow} (B_R(x),d) $, and $ x_i\to x $ as $ i\to \infty $. Lemma \ref{group-GH_converge} and Lemma \ref{quotient-GH-converge} also hold for the pointed Gromov-Hausdorff convergence.
	
	Let $ (X_i,x_i,d_i)\overset{GH}{\longrightarrow} (X,x,d) $, $ \Gamma_i $ be the fundamental group of $ (X_i,x_i) $, and $ (\tilde{X}_i,\tilde{x}_i) $ be the universal covering space of $ (X_i,x_i)$. If $ (\tilde{X}_i,\tilde{x}_i) $ Gromov-Hausdorff converges to a length metric space $ (\tilde{X},\tilde{x}) $, i.e., $ (\tilde{X}_i,\tilde{x}_i) \overset{GH}{\longrightarrow} (\tilde{X},\tilde{x}) $, then by Lemma\ref{group-GH_converge} and Lemma \ref{quotient-GH-converge}, we have the following commutative diagram
	$$\begin{CD}
		(\tilde{X}_i,\tilde{x}_i ,\Gamma_i) @>GH>> (\tilde X,\tilde{x}, G)\\
		@V\pi_iVV @V\pi VV\\
		(X_i,x_i) @>GH>> (X,x),
	\end{CD}
	$$
	where $ \pi_i: (\tilde{X}_i,\tilde{x}_i, \Gamma_i) \to (X_i,x_i) $ is the universal cover, $ \Gamma_i $ is the deck-transformation group, and $ \pi $ is the limit map of $ \pi_i $.
	
	There is a precompactness of complete Riemainnian $ n $-manifolds with Ricci curvature bounded below.
	\begin{theorem}[Gromov's precompactness \cite{GLP1981}]
		Let $ (M_i,p_i,g_i) $ be a sequence of complete Riemainnian $ n $-manifolds satisfying 
		$$ \op{Ric}_{g_i}\geq (n-1)H,$$
		where $ H\in \mathbb{R} $ is a constant. Then $ (M_i,p_i,g_i) $ Gromov-Hausdorff converges to a length metric space after passing to a subsequence.
	\end{theorem}

	\subsection{$ C^{\infty} $-convergence of Einstein metrics}
	In this subsection, we first give the definition and properties of $ C^{k,\alpha} $-convergence. Next, we recall  the $ C^{\infty} $-convergence of Einstein metrics.
	
	\begin{definition}
		Let $ M $ be a compact smooth $ n $-manifold. Let $ g_i $ and $ g $ be complete $ C^{k,\alpha} $-smooth Riemannian metrics on $ M $, where $ \alpha\in (0,1] $. We say that $ g_i $ converges to $ g $ in the sense of $ C^{k,\alpha} $-norm if for any $ p\in M $, there exist a coordinate chart around $ p $, $ (x_1,\cdots,x_n): U \to \Omega \subset \mathbb{R}^n$, such that $ g_{i,st}=g_i(\frac{\partial}{\partial x_s},\frac{\partial}{\partial x_t}) $ $ C^{k,\alpha} $-converges to $ g_{st}=g(\frac{\partial}{\partial x_s},\frac{\partial}{\partial x_t}) $ as $ i \to \infty $, i.e., 
		$$ \| g_{i,st}-g_{st} \|_{C^{k,\alpha}(\Omega)} \to 0, \text{ as } i\to \infty ,$$
		where the $ C^{k,\alpha} $-norm $ \|f\|_{C^{k,\alpha}(\Omega)} $ of a smooth function is defined by 
		$$  \|f\|_{C^{k,\alpha}(\Omega)}=\sup_{x\in\Omega} |f(x)| +\sum_{|I|=1}^{k}\sup_{x\in\Omega} | \partial^I f(x) |_{C^{|I|}}+\sum_{|I|=k}\sup_{x,y\in \Omega}\frac{|\partial^I f(x)-\partial^I f(y)|}{|x-y|^{\alpha}} ,$$
		where $ \partial^I f(x) =\frac{\partial^{|I|}f}{ \partial (x_1)^{i_1}\cdots\partial (x_n)^{i_n}} $, and $ I=(i_1,\cdots,i_n) $, $ |I|=i_1+\cdots+i_n $.
	\end{definition}
	We say that $ g_i $ converges to $ g $ in the sense of $ C^\infty $-norm if $ g_i $ converges to $ g $ in the sense of $ C^k $-norm, for all $ k>0 $.
	
	\begin{definition}
		Let $ (M_i,g_i) $ and $ (M,g) $ be compact smooth $ n $-manifolds with $ C^{k,\alpha} $-smooth Riemannian metrics. We say that $ (M_i,g_i) $ converges to $ (M,g) $ in the $ C^{k,\alpha} $-topology if there exists an integer $ i_0>0 $ such that the following holds: for each $ i\ge i_0 $ there exists $ C^{k+1,\alpha} $-diffeomorphism $ \Phi_i: M\to M_i $ such that the pullback metric $ \Phi_i^*g_i $ converges to $ g $ in the sense of $ C^{k,\alpha} $-norm.
		
		We say that $ (M_i,g_i) $ converges to $ (M,g) $ in the $ C^{\infty} $-topology if $ (M_i,g_i) $ converges to $ (M,g) $ in the $ C^k $-topology, for all $ k>0 $.
	\end{definition}
	
	The Anderson's $ C^{1,\alpha} $-convergence theorem provides a criterion that a Gromov-Hausdorff Cauchy sequence of manifolds converges in $ C^{1,\alpha} $-topology.
	
	\begin{theorem}[$ C^{1,\alpha} $-convergence theorem,\cite{Anderson1990}]\label{Thm-convergence}
		Assume $ (M_i,g_i)\overset{GH}{\longrightarrow}(X,d) $, where$ (M_i,g_i) $ is a compact smooth $ n $-manifold satisfying 
		$$ |\op{Ric}|_{g_i}\le n-1,\quad \op{diam}(M_i,g_i)\le D,\quad \op{Vol}(B_1(p_i))\ge v>0 ,$$
		and any point $ x\in X $ is regular, i.e., any tangent cone of $ x $ is $ \mathbb{R}^n $. Then $ (X,d) $ is isometric to a $ C^{1,\alpha} $-Riemannian manifold and $ (M_i,g_i) $ converges to $ (X,d) $ in the $ C^{1,\alpha} $-topology, where $ \alpha\in (0,1) $.
		
		If, in addition, $ (M_i,g_i) $ is Einstein, i.e., $ \op{Ric}_{g_i}=\lambda_i g_i $, then $ (X,d) $ is isometric to a Einstein manifold, and $ (M_i,g_i) $ converges to  $ (X,d) $ in the $ C^{\infty} $-topology.
	\end{theorem}
	
	In fact, for any $ Q>0 $ and $ p_i\in M_i $ converges to $ x\in X $ in Theorem \ref{Thm-convergence}, there exists $r_h(x,Q,\alpha)>0 $ such that there is a harmonic coordinate chart $ (x_1,\cdots,x_n):(B_{r_h}(p_i),g_i)\to \Omega\subset \mathbb{R}^n $, and the metric tensor of $ g_i $ satisfying $ e^{-Q}\delta_{st} \le g_{i,st} \le e^{Q}\delta_{st} $ and $ r_h^{1+\alpha}\| \partial g_{i,st} \|_{C^{0,\alpha}(\Omega)}\le e^Q $, where $ \delta_{st} $ is Kronecker symbols. The largest $ r_h $ on $ (M,g) $ is called the $ C^{1,\alpha} $-harmonic radius of $ (M,g) $.

	if $ (M_i,g_i) $ is Einstein, then the metric tensor $ g $ in the harmonic coordinate chart above satisfies $ \Delta =\sum_{k,l=1}^{n}g_i^{kl}\frac{\partial^2 }{\partial x_k \partial x_l} $ and
	$$ \frac{1}{2}\Delta g_{i,st} = -\lambda g_{i,st}+ Q(g_i,\partial g_i), $$
	where each item of $ Q $ is a fraction, the denominator depends only on $ \sqrt{\op{det}g_{i,st}} $ and the numerator is quadratic in $ \partial g_i $ times some $ g_{i,kl} $ in the matrix $ g_i $, and $ \Delta $ is the Laplacian of $ g_i $. By the standard Schaulder estimate for elliptic PDE,
	$ (M_i,g_i) $ converges to $ (X,d) $ in the $ C^\infty $-topology. Hence, $ (X,d) $ is also a Einstein manifold.
	
	\subsection{ Rigidity of Einstein metrics under sectional curvature bounds}
	The references of this part is \cite[Chapter 12]{Besse1987}.
	
	An Einstein structure is an isometric equivalence class of Riemannian metrics on a given compact smooth manifold $ M $.
	
	\begin{definition}
		The moduli space of Einstein structures on $ M $, denoted $ \mathcal{E}(M) $, is the quotient 
		$$ \{\text{Einstein metrics with total volume } 1\}/\op{Diff}(M) ,$$
		where $ \op{Diff}(M)=\{ \varphi|\varphi \text{ is a diffeomorphism of }M \} $, $ \mathcal{E}(M) $ endowed with the quotient topology.
	\end{definition}
	
	\begin{definition}
		An Einstein structure of $ M $ is called rigid if it is an isolated point of the Moduli Space.
	\end{definition}
	
	The following rigidity of Einstein metrics under bounded sectional curvature will be applied later in the proofs of Theorems \ref{thm-gap}, \ref{thm-rigid-localrewindingvol}.
	
	\begin{theorem}[{\cite[12.72 Corollary]{Besse1987}}]\label{Thm-pinch-curvature-rigid}
		Let $ M $ be a smooth $ n $-manifold.
		Any Einstein structure of $ M $ with $ \delta $-pinched sectional curvature, $ \frac{n-2}{3n}< \delta\le 1 $, is rigid.
	\end{theorem}
	
	\begin{theorem}[{\cite[12.73 Corollary]{Besse1987}}]\label{Thm-neg-curvature-rigid}
		Let $ M $ be a smooth $ n $-manifold, $ n\ge 3 $.
		Any Einstein structure of $ M $ with negative sectional curvature is rigid.
	\end{theorem}
	
	\section{Proof of Main Theorem}\label{preliminaries}
	\begin{proof}[Proof of Theorem \ref{thm-rigidity-almostflat}]
		~
		
		By Theorem \ref{thm-almostflat}, (1), (2) and (3) are equivalent. It suffices to show(2) and (3) implies (4).
		
		Step 1. $ (M,g) $ satisfying (2) and (3) must be Ricci-flat.
		
		Let us argue by contradiction. Suppose that there is a sequence of Einstein $n$-manifolds $(M_i,g_i)$ such that $\op{Ric}_{g_i}=\lambda_ig$ with $\lambda_i\ge -(n-1)$ and $\op{diam}(M_i,g_i)\to 0$, and satisfying (2) and (3) in Theorem \ref{thm-almostflat}, but none of $ \lambda_i \ne 0 $. By (2) and Bonnet-Myer's theorem, $\lambda_i< 0$. Then up to a rescaling on the metrics, we assume that $\lambda_i=-(n-1)$. Note that $ (M_i,g_i) $ still Gromov-Hausdorff converges to a point. By Theorem \ref{thm-almostflat}, (3) still holds. That is, the Riemannian universal cover $ (\tilde{M}_i,\tilde{x}_i) $ is a non-collapsing sequence.
		
		By passing to a sequence, let us consider the following equivariant Gromov-Hausdorff convergence,
		$$\begin{CD}
			(\tilde{M}_i,\tilde x_i, \Gamma_i) @>GH>> (\tilde X, \tilde x, G)\\
			@V\pi_iVV @V\pi VV\\
			(M_i,x_i) @>GH>> \op{pt}
		\end{CD}
		$$
		where $\Gamma_i$ is the deck-transformation of fundamental group $\pi_1(M_i,x_i)$ and $G$ its limit group. By the generalized Margulis lemma \cite{KW} and Cheeger-Colding \cite{CC1997I}, the identity component $ G_0 $ of $ G $ is a nilpotent Lie group.
		
		Since $\tilde X$ is a non-collapsed Ricci limit space, by \cite{CC1997I} regular points are dense in $ \tilde{X} $, where any tangent cone is isometric to $\mathbb R^n$. Because $G$ acts on $\tilde X$ transitively, all points in $\tilde X$ are regular. Hence, by Theorem \ref{Thm-convergence}, $(\tilde M_i, \tilde x_i)$ converges to $\tilde X$ in the $C^\infty$-norm on every $R$-ball (cf. \cite[Theorem 7.3]{CC1997I}), and $\tilde X$ is a smooth Einstein Riemannian manifold $(\tilde X, \tilde g)$ with Einstein constant $\lambda_\infty=\lim_{i\to\infty}\lambda_i= -1$.
		
		Claim: the identity component $G_0$ acts on $\tilde X$ transitively and freely.
		
		By the claim, $(\tilde X, \tilde g)$ is a nilpotent Lie group with a left invariant metric. 
		%On the other hand, $\op{Ric}_{\tilde g}=\lambda_\infty \tilde g$, where $\lambda_\infty=\lim_{i\to \infty}\lambda_i$. 
		By Milnor \cite[Theorem 2.4]{Milnor1976}, any left invariant metric of a nilpotent but not 
		commutative group has both directions of strictly negative Ricci curvature and positve Ricci curvature. Hence $(\tilde X, \tilde g)$ must be a flat manifold. That is, $ \lambda_\infty=0 $, a contradication. (In fact, by Theorem \ref{thm-almostflat} or \cite[Proposition 5.8]{NaberZhang2016}, it can be seen $(\tilde X,\tilde g)$ is isometric to $\mathbb R^n$. But we do not need this here.)
		% By Cheeger-Gromoll's splitting theorem, $\tilde X$ isometrically splits to $\mathbb R^k\times Y$, where $Y$ must be compact. Since the nilpotent group $G_0$ acts transitively, $Y$ must be a torus $T^m$. 
		
		Step 2. Because $ (M,g) $ is Ricci-flat, then either (2) or (3) implies the Riemannian universal cover $ \tilde{M}=\mathbb{R}^n$.There are two methods to illustrate.
		
		Method 1. By Cheeger-Gromoll's splitting theorem \cite{Cheeger-Gromoll1971} each of the universal cover of $ M_i $ isometrically splits to $\mathbb R^{k}\times Y$, where $Y$ is compact and simply connected. Moreover, there is a finite cover $ \hat{M} $ of $ M $ which is diffeomorphic to $ T^k \times Y $ by \cite[Theorem 9.2]{Cheeger-Gromoll72}. Because the polycyclic rank of $ \pi_1(M) $ is $ n $, it follows that $ Y $ is a point and $ \tilde{M}=\mathbb{R}^n $.
		
		%Now a contradiction can be derived by dividing into the following two cases.
		
		Method 2. Let us argue by contradiction. There are infinite many $(M_i,g_i)$ which are Ricci-flat but non-flat satisfying (3) and $ \op{diam}(M_i)\to 0 $. By Cheeger-Gromoll's splitting theorem each  universal cover of $ M_i $ isometrically splits to $\mathbb R^{k_i}\times Y_i$, where $Y_i$ is compact and simply connected, but not a point. 
		
		%First, let us follow the proof of \cite[Theorem C]{CRX2019} to show that the diameter of $Y_i$ is uniformly bounded. 
		Since $ \tilde{M}_i $ is not collapsing, $ Y_i $ can not converges to a point, i.e., $ \op{diam}(Y_i)\ge d>0 $.
		Let $d_i=\op{diam}(Y_i) $, then the rescaled manifolds $d_i^{-1}(M_i,g_i)$ satisfies $ \op{diam}(Y_i)=1 $, and their universal covers admit a equivariantly Gromov-Hausdorff convergent subseqence, i.e.,
		$$\begin{CD}
			(d_i^{-1}\mathbb R^{k_i}\times Y_i,\tilde x_i, \Gamma_i) @>GH>> (\mathbb R^{k_\infty}\times Y_\infty, \tilde x, G)\\
			@V\pi_iVV @V\pi VV\\
			(d_i^{-1}M_i,x_i) @>GH>> \op{pt}
		\end{CD}
		$$
		Since the nilpotent group $G_0$ acts transitively, $Y_\infty$ must be a torus $T^m$.  Since $d_i^{-1} Y_i$ Gromov-Hausdorff converges to $Y_\infty$, there is an onto map from $\pi_1(Y_i)$ to $\pi_1(T^m)$ (cf. \cite{Tuschmann1995}), a contradiction.

		%Secondly, by passing a subsequence, the original universal covers $\tilde M_i=\mathbb R^{k_i}\times Y_i$ Gromov-Hausdorff converges to $(\tilde X,\tilde g)$. It follows that $\tilde X$ is isometric $\mathbb R^{k_\infty}\times Y$, where $Y$ is a compact flat manifold and $Y_i$ Gromov-Hausdorff converges to $Y$. By the claim again, $Y$ is a torus. Since $Y_i$ is simply connected, the same contradiction as above is derived.
		
		%Case 2. None of $(M_i,g_i)$ is Ricci-flat. Then up to a rescaling on the metrics, we assume that $\lambda_i=-(n-1)$. Note that $ (M_i,g_i) $ still Gromov-Hausdorff converges to a point. By Theorem \ref{thm-almost-rigid}, (4) still holds. Repeating the argument about the equivariant Gromov-Hausdorff convergence $ (\tilde{M}_i,\tilde x_i, \Gamma_i) \overset{GH}{\longrightarrow} (\tilde X, \tilde x, G) $ above (see the paragraphs before case 1), the limit space of $\tilde M_i$ is still flat, which implies that $\lambda_i\to 0$, a contradiction.
		
		Proof of the claim: 
		
		Let us consider the component $G_0$ of the identity in $G$. Since $\tilde X/G_0$ is $0$-dimensional and connected, $G_0$ acts on $\tilde X$ transitively.
		Since $G_0$ is a nilpotent Lie group, for any $y\in \tilde X$, the isotropy group $G_{0,y}$ of $G_0$ at $y$, which is compact, lies in the center of $G_0$. Therefore $G_{0,y}$ is normal in $G_0$. And thus $G_{0,y}$ acts trivially on $\tilde X$. Combing the fact that the action of $G$ on $\tilde X$ is effective, we see that $G_{0,y}$ must be trivial. 
	\end{proof}
	
	Theorem \ref{thm-gap} is a corollary of Theorem \ref{thm-almost-rigid}.

	\begin{proof}[Proof of Theorem \ref{thm-gap}]
		\quad
		
		Let us argue by contradiction. Suppose that there is a sequence of Einstein $n$-manifolds $(M_i,g_i)$, $\op{Ric}_{g_i}=\lambda_i g_i$, $n\ge 3$ satisfying
		$$\lambda_i\ge -(n-1),\quad h(M_i,g_i)> n-1-\epsilon_i, \quad \operatorname{diam}(M_i,g_i)\le D,$$
		but none of $(M_i,g_i)$ is hyperbolic.
		
		First, by Theorem \ref{thm-almost-rigid} and Theorem \ref{Thm-convergence} (c.f. \cite[Theorem 7.3]{CC1997I}), by passing to a subsequence $(M_i,g_i)$ are diffeomorphic to a hyperbolic manifold $\mathbb H^n/\Gamma$, and $g_i$ converges to $g$ in the $C^{\infty}$-topology.
		
		Secondly, by Theorem \ref{Thm-neg-curvature-rigid} for $n\ge 3$, any Einstein structure with negative sectional curvature is rigid. Hence $g_i$ is also hyperbolic for all large $i$. Thus a contradiction is derived.
	\end{proof}
	
	\begin{proof}[Proof of Theorem \ref{thm-rigid-localrewindingvol}]
		\quad
		
		Assume that the Ricci curvature of $ (M,g) $ satisfies $ \op{Ric}_g=\lambda g $. Let $ x\in M $ be a fixed point. Since $ \frac{\op{vol}B_\rho (x^*)}{\op{vol} B_\rho^H}\ge 1-\epsilon$, by Bishop volume comparison $ \frac{\lambda}{n-1}$ is $\delta(\epsilon|n,\rho)$-close to $ H $, where $ \delta\to 0 $ when $ \epsilon \to 0 $. 
		Moreover, by Cheeger-Colding almost maximal volume theorem \cite{Cheeger-Colding96}, $ B_{\rho/2} (x^*) $ is Gromov-Hausdorff close to $B_{\rho/2}^H $ for $ \epsilon $ small. By Anderson convergence Theorem \ref{Thm-convergence}, there is a diffeomorphism from $ B_{\rho/2}^H $ to $ B_{\rho/2} (x^*) $ such that the pull back metric is $ C^{1,\alpha} $ close to that of $ B_{\rho/2}^H $. Because the pull back metric is Einstein, they are $ C^\infty $ close to each other. In particular, the sectional curvature of $ M $ satisfies $ H-\varkappa(\epsilon|n,\rho) \le\op{Sec}_g\le H+\varkappa(\epsilon|n,\rho) $ (c.f. \cite[Lemma 3.6]{CRX2019}). 
		
		Case 1. $ |H| $ sufficient small.
		
		If $ H=0 $, then the manifold $ (M,g) $ satisfies $ \op{diam}(M,g)^2\cdot\max|\op{Sec}_g|< D^2\varkappa(\epsilon|n,\rho)$. By Corollary \ref{cor-almost-flat-Einstein-mfd}, there is $ \epsilon(n,\rho,D)>0 $ depending only on $ n,\rho,D $  such that for any $ 0<\epsilon\le\epsilon(n,\rho,D) $, $ M $ is flat.
		
		Observe that for $ -\epsilon(n,\rho,D)/2<H(n-1)<\epsilon(n,\rho,D)/2  $ and $ 0\le\epsilon<\epsilon(n,\rho,D)/2 $, $ M $ satisfies 
		$$\op{Ric}_g\ge -\epsilon(n,\rho,D),\quad \op{diam}(M,g)\le D, \quad \frac{\op{vol}B_\rho (x^*)}{\op{vol} B_\rho^0}\ge (1-\epsilon)\frac{\op{vol}B_\rho^H}{\op{vol}B_\rho^0},$$
		where $\frac{\op{vol}B_\rho^H}{\op{vol}B_\rho^0} \to 1$, as $ H\to 0 $.
		Hence, there is $ H(n,\rho,D)>0 $ and $ \epsilon(n,\rho,D)>0 $ such that for any $ |H| < H(n,\rho,D), \epsilon<\epsilon(n,\rho,D) $, $ M $ is also flat.
		
		What remains is to show the case that $ |H|>H(n,\rho,D) $.
		Up to a rescaling of the metric, we assume $ H=\pm 1 $ in the following.
		
		Case 2. $  H = -1 $. 
		
		Since $ M $ has bounded negative sectional curvature $ |\op{Sec}_g+1|\le \varkappa(\epsilon|n,\rho) $, by Heintze-Margulis lemma \cite{Heintze} the volume of $ M $ has a lower bounded $ \op{vol}(M)\ge v(n)>0 $. By Cheeger-Gromov convergence theorem \cite{Cheeger67,Petersen} and the fact that $ M $ is Einstein, $ (M,g) $ is diffeomorphic to a hyperbolic manifold whose metrics are $\varkappa(\epsilon|n,\rho,D)$ $C^\infty$ close.
		By Theorem \ref{Thm-neg-curvature-rigid}, any Einstein structure with negative sectional curvature is rigid. It tends out that $ M $ is  isometric to a hyperbolic manifold for $ \epsilon $ sufficient small.
		
		Case 3. $ H=1 $. 
		
		Because the curvature of $ M $ is almost $ 1 $, the Klingenberg's $ 1/4 $-pinched injectivity radius estimate implies that the injective radius of $ \tilde{M} $ is $ \ge \pi/2 $. By Cheeger-Gromov convergence Theorem \ref{Thm-convergence} and the fact that $ M $ is Einstein, the simply connected manifold $ \tilde{M} $ is  diffeomorphic to $ \mathbb{S}^n $ whose metrics are $ C^{\infty} $ close. By Theorem \ref{Thm-pinch-curvature-rigid}, any Einstein structure with $ \delta $-pinched sectional curvature with $ \delta>(n-2)/3n $ is rigid. It tends out that $ \tilde{M} $ is  isometric to $ \mathbb{S}^n $ for $ \epsilon $ sufficient small.
	\end{proof}
	
	\section{Stratified Einstein spaces}\label{startified-space}
	
	The references of this section are \cite{Honda-Mondello}, \cite{BKMR2021}, \cite{ALMP12}.
	
	A $ n $-dimensional stratified space \cite{BKMR2021},\cite{ALMP12} is a metrizable, locally compact, second countable topological space $ X $ with a fixed stratification, which is a locally finite decomposition into a union of strata
	\begin{equation}\label{def-stratification}
		X=\bigsqcup_{l_j=l_1}^{l_k}\Sigma^{l_j}(X) , 0\le l_1<l_2<\cdots<l_{k-1}<n-1<l_k=n 
	\end{equation}
	satisfying the following three conditions.
	\begin{enumerate}
		\item All $\Sigma^{l_j}(X) $ are smooth, possibly open, $ l_j $-manifolds, where $ l_{k-1}\le n-2 $, i.e, $ \Sigma^{n-1}(X)=\emptyset $, and for any two components  $Y_\alpha$, $Y_\beta$ of all strata, $Y_\alpha \subset \overline{Y_\beta}$ if $Y_\alpha \cap \overline{Y_\beta} \neq \emptyset$.
		
		\item Each component of $\Sigma^{l}(X) $ corresponds with a compact and connected stratified space $ Z_l $ of dimension $ n-l-1 $ such that for any $ x $ in a component of $\Sigma^{l}(X) $, there is a homeomorphic $ \varphi_x $ from $ B_{r_x}(0_l)\times C_{[0,r_x)}(Z_l) $ to the neighborhood $ U_x $ of $ x $ preserving the stratification between $ B_{r_x}(0_l)\times C_{[0,r_x)}(Z_l)$ and $ U_x $, where $ B_{r_x}(0_l) $ is a Euclidean ball in $ \mathbb{R}^l $ and $ C_{[0,r_x)}(Z_l) $ is a truncated cone over $ Z_l $, i.e, the quotient space $ ([0,r_x)\times Z_l)/\sim$ with the equivalence relation $ (0, z_1) \sim (0, z_2) $ for all $ z_1, z_2\in Z_l $. $ Z_l $ is called the link of component of $ \Sigma^{l}(X) $.
		\item The local trivialization $ \varphi_x $, $ x\in \Sigma^l(X) $, can be glued together to a global trivialization of a neighborhood of each component of $ \Sigma^l(X) $. 
	\end{enumerate}
	
	\begin{remark}

		There is a natural stratification on $  \mathbb{R}^l\times C(Z_l) $ induced by $ Z_l $, that is, $ \Sigma^l(\mathbb{R}^l\times C(Z_l))= \mathbb{R}^l\times \{o \}$ and $ \Sigma^{l+k}(\mathbb{R}^l\times C(Z_l))= \mathbb{R}^{l}\times (0,+\infty) \times \Sigma^{k-1}(Z_l)$ for $ k\ge 1 $. The codimension of singular stratum $  \Sigma^l(\mathbb{R}^l\times C(Z_l)) $  (definition see below) is $ n-l $, and the others are the same as that of $ Z_l $.
	\end{remark}
	
	In this paper, we only consider complete stratified spaces with a length metric.
	
	By Condition (2) above, $ \Sigma^{n}(X) $ is open and dense in $ X $, which exclude the spaces such as the union of $ \mathbb{S}^3 $ and a segment at a common point. We call $ \Sigma^{n}(X) $ regular set, and its complement $ \Sigma(X)=\bigsqcup_{l_j=l_1}^{l_{k-1}}\Sigma^{l_j}(X) $ singular set, where $ \Sigma^{l}(X) $ is called the singular stratum of dimension $ l $ for $ l \le n-2 $. 
	By definition $ Z_l $  is connected, whose dimension $ n-l-1 $ is always no less than $ 1 $ by the assumption $ \Sigma^{n-1}(X)=\emptyset $. Hence, the spaces such as the union of two spheres $ \mathbb{S}^3 $ at a point are excluded.
	
	A stratified space $ X $ is called smooth if for any $ x\in \Sigma^l(X) $, the local chart $ \varphi_x $ is a smooth diffeomorphism over regular set $ (B_{r_x}(0_l)\times C_{[0,r_x)}(\Sigma^{l}(Z_l)))\setminus B_{r_x}(0_l)\times\{0\}$.

	\begin{example}\label{ex-stratified-space}
		Since $ \mathbb{R}^3=\mathbb{R}\times C(\mathbb{S}^1)=C(\mathbb{S}^2) $, the topological space $ X= \mathbb{R}^3 $ equipped with its canonical smooth structure admits many different stratification. 
		
		One is formed by $ \Sigma_A^3(X)=\left(\mathbb{R}\times C(\mathbb{S}^1)\right)\setminus \left(\mathbb{R}\times\{0\} \right) $ and $ \Sigma_A^1(X)=\mathbb{R}\times\{0\} $. Another consists of $ \Sigma_B^3(X)=C(\mathbb{S}^2)\setminus \left( \{o\} \right) $ and $ \Sigma_B^0(X)= \{o\} $. Both of the two stratified spaces $ X_A=(X,\Sigma_A^3(X)\bigsqcup \Sigma_A^1(X)) $ and $ X_B=(X,\Sigma_B^3(X)\bigsqcup \Sigma_B^0(X)) $ are smooth, but they are different.
		
		The link of $ X_A $ at any singular point is $ \mathbb{S}^1 $ and the link of $ X_B $ at its unique singular point is $ \mathbb{S}^2 $ which is the join $ \mathbb{S}^0*\mathbb{S}^1 $, where $ \mathbb{S}^0=\{0,\pi\} $ and the join $ X*Y $ of two topological spaces $ X,Y $ is the quotient space $ \left([0,\frac{\pi}{2}] \times X\times Y\right)/\sim $, where $ (0,x,y_1)\sim (0,x,y_2) $ and $ (\frac{\pi}{2},x_1,y)\sim (\frac{\pi}{2},x_2,y) $.
		
		Note that the join $ X*Y $ of two metric spaces $ (X,d_X) $ and $ (Y,d_Y) $ can be realized as a canonical spherical join $[0,\frac{\pi}{2}] \times_{\cos} X\times_{\sin} Y$ where the distance is defined by $$ \cos d\left((a_1, x_1, y_1), (a_2,x_2,y_2)\right)= \cos a_1  \cos a_2\cos d_X(x_1,x_2) + \sin a_1 \sin a_2 \cos d_Y(y_1,y_2) .$$ 
		Similarly, the cone $ C(X) $ over a metric space $ (X,d_X) $ can be realized as a Euclidean cone whose metric is defined by $$ d\left((r_1,x_1),(r_2,x_2)\right)^2=r_1^2+r_2^2-2r_1r_2\cos\left(\min\{d_X(x_1,x_2),\pi\}\right) .$$  
		Then the Euclidean cone $ C(X*Y) $ is isometric to $ C(X)\times C(Y) $ (c.f. \cite{Burago-Gromov-Perelman1992}).
	\end{example}

	An iterated edge metric $ g $ on a smooth stratified space $ X $ is a Riemannian metric on the regular set of $ X $, which is iteratively defined such that on each local chart around a singular point, $ g $ is appropriate asymptotically close to a warped product metric, i.e., 
	there exists constant $ C,\alpha>0 $ such that for each $ x$ in a component of $\Sigma^l(X) $ and  for some $ 0<\delta_x\le r_x $,
	\begin{equation}\label{def-iterated-metric}
		|\varphi_x^*g-(g_0+dr^2+r^2k_l)|\le Cr^\alpha \text{ on } (B_{\delta_x}(0_l)\times C_{[0,\delta_x)}(\Sigma^{l}(Z_l)))\setminus B_{\delta_x}(0_l)\times\{0\}, 
	\end{equation}
	where $ g_0 $ is a standard Riemannian metric on $ \mathbb{R}^l $, and $ k_l $ is an iterated edge metric on $ Z_l $. For $ \forall x\in  \sum^l(X)  $, there exists a unique tangent cone $ C(S_x) $ over the tangent sphere of $ x $ which is isometric to $ (\mathbb{R}^l \times C(Z_l),g_0+dr^2+r^2k_l) $.

	By definition, any iterated edge metric $ g $ is smooth on the regular stratum. But \eqref{def-iterated-metric} dose not implies that $ g $ is not smooth at a singular point. In fact, $ (X_A,g_B) $ in the following Example \ref{ex-iterated-metric} is smooth at most of the singular points.
	
	\begin{example}\label{ex-iterated-metric}
		We construct two iterated edge metrics $ g_A$, $ g_B $ on the stratified spaces $ X_A $ in Example \ref{ex-stratified-space}. 
		
		Let us first view $ X_A=\mathbb{R}^1\times C(\mathbb{S}_{\frac{\sqrt{2}}{2}}^1) $ equipped with the canonical Euclidean cone metric $ g_A=dx^2+d\rho^2+\frac{\rho^2}{2}d\theta^2 $, where $ x\in \mathbb{R}^1, (\rho,\theta)\in C(\mathbb{S}_{\frac{\sqrt{2}}{2}}^1), \theta \in [0,2\pi] $. 
		
		Next, we construct another iterated edge metric $ g_B $ on $ X_A $ by smoothing $ g_A $ such that for any $ x\ne0 $, each cone $ \{x\}\times C(\mathbb{S}_{\frac{\sqrt{2}}{2}}^1) $ is equipped with a smooth metric $ h_x$, which approaches $ h_0=d\rho^2+\frac{\rho^2}{2}d\theta^2 $ as $ x\to 0 $, and the blow up of $ g_B $ at $ (0,o)\in \mathbb{R}^1\times C(\mathbb{S}_{\frac{\sqrt{2}}{2}}^1) $ is $ g_A $, where $ o $ is the vertex of $ C(\mathbb{S}_{\frac{\sqrt{2}}{2}}^1) $.
		In fact, $ X_A $ can be embedded to $ \mathbb{R}^4 $ as a generalized parabolic hyper-surface $ x_4=\sqrt{x_1^4+x_2^2+x_3^2} $, where $ (x_1,x_2,x_3,x_4)\in \mathbb{R}^4 $ satisfies $ x_1=x $, $ x_2=r\cos \theta $, $ x_3=r\sin \theta $. Then, the pull back metric is $ g_B $ whose restriction on each cone $ \{x\}\times C(\mathbb{S}_{\frac{\sqrt{2}}{2}}^1) $ is $ h_x=d\rho^2+r^2(\rho,x)d\theta^2 $,  where $ \rho=\int_{0}^{r}\sqrt{1+\left(\partial_r(\sqrt{r^2+x^4})\right)^2}\op{d}r $. After blowing up the generalized parabolic hyper-surface at the origin in $ \mathbb{R}^4 $, it becomes $ (\mathbb{R}^1\times C(\mathbb{S}_{\frac{\sqrt{2}}{2}}^1),g_A) $.
		
		Furthermore, because $ x_4=\sqrt{x_1^4+x_2^2+x_3^2} $ is a convex function, $ (X_A,g_B) $ is an Alexandrov space with curvature $ \ge 0 $ which contains only one singular point $ (0,0,0,0) $. The tangent cone of $ (X_A,g_B) $ at $ (0,0,0,0) $ is isometric to $ (X_A,g_A) $. It is clear that $ g_B $ is an iterated edge metric on $ X_A $, which is smooth on singular stratum $ \Sigma_A^1(X)\setminus \left( \{0\}\times\{0\} \right) $.
		
		In the following, we show that $ g_B $ is not an iterated edge metric on $ X_B $.
		Observe that, $ \mathbb{R}^1\times C(\mathbb{S}_{\frac{\sqrt{2}}{2}}^1)=C(S(\mathbb{S}_{\frac{\sqrt{2}}{2}}^1)) $, where $ S(\mathbb{S}_{\frac{\sqrt{2}}{2}}^1) $ is the suspension of $ \mathbb{S}_{\frac{\sqrt{2}}{2}}^1 $, i.e., $ \left([0,\pi]\times \mathbb{S}_{\frac{\sqrt{2}}{2}}^1\right)/ \sim $, $ (0,y_1) \sim(0,y_2) $ and $ (\pi,y_1) \sim(\pi,y_2) $, and the metric $$ g_A=dx^2+d\rho^2+\frac{\rho^2}{2}d\theta^2= d\alpha^2+\alpha^2 \left(d\beta^2+\frac{\sin^2\beta }{2}d\theta^2\right) ,$$ where  $ x=\alpha\cos \beta $, $ \rho=\alpha\sin \beta $. Hence, $ g_B $ is appropriate asymptotically close to $ g_A $ at $ \{0\}\times\{0\} $ on $ C(S(\mathbb{S}_{\frac{\sqrt{2}}{2}}^1))\setminus \mathbb{R}^1\times \left(\{0\}\times\{0\}\right) $, where $ \{0\}\times\{0\} $ is the unique singular point in $ X_B $. The link of $ \Sigma_B^0(X) $ is $ S(\mathbb{S}_{\frac{\sqrt{2}}{2}}^1) $ whose strata are $ \Sigma^2(S(\mathbb{S}_{\frac{\sqrt{2}}{2}}^1))=S(\mathbb{S}_{\frac{\sqrt{2}}{2}}^1)\setminus \{(0,y),(\pi,y)\} $ and $ \Sigma^0(S(\mathbb{S}_{\frac{\sqrt{2}}{2}}^1))=\{(0,y),(\pi,y)\} $, where the iterated edge metric on the link is $ d\beta^2+\frac{\sin^2\beta }{2}d\theta^2 $. It follows that the natural induced stratification on $ C(S(\mathbb{S}_{\frac{\sqrt{2}}{2}}^1)) $ has $ 1 $-dimensional singular stratum. Therefore, the trivialization of $ X_B $ does not preserve the strata of $ X_B $, and $ g_B $ is not an iterated edge metric on $ X_B $.
		
		Meanwhile, $ g_A $ is not an iterated metric on $ X_B $, because $ g_A $ is not smooth at those regular points in $ (0,+\infty)\times \{(0,y),(\pi,y)\}\subset \Sigma^3(X_B) $.
	\end{example}
	
	Let $ (X,g) $ be a compact stratified $ n $-space $ X $ with an iterated edge metric $ g $. For any $ x\in \Sigma^{n-2}(X) $, if it is no empty, the tangent cone $ C(S_x) $ is isometric to $ (\mathbb{R}^{n-2} \times C(\mathbb{S}^1),g_0+dr^2+r^2(\frac{\alpha_x}{2\pi})d\theta^2) $, where $ \alpha_x\in (0,+\infty) $ is called the angle of $ \Sigma^{n-2}(X) $ at $ x $. 
	
	By Bertrand-Ketterer-Mondello-Richard \cite[Theorem A]{BKMR2021}, if the Ricci curvature satisfies $ \op{Ric}_g\ge K $ on the regular set and $ \alpha_x\le 2\pi $ for any $ x\in\Sigma^{n-2}(X) $, then $ (X,d_g,\mathcal{H}^n) $ is a $ RCD(K,n) $ metric measure space, where $ d_g $ is the distance endowed by $ g $, and $ \mathcal{H}^n $ is the $ n $-dimensional Hausdorff measure.
	
	\begin{definition}\label{def-Einstein-stratified-space}
		A $ n $-dimensional stratified space $ (X,g) $ with an iterated edge metric $ g $ is called \emph{$ \lambda $-Einstein}, if it satisfies the following two conditions:
		\begin{enumerate}
			\item $ \op{Ric}_g=\lambda g $ on the regular set for some $ \lambda\in \mathbb{R} $;
			\item  $ \Sigma^{n-2}(X)=\emptyset  $.
			%The Hausdorff dimension $ \op{dim}_\mathcal{H}(\mathcal{S}(X,g))< n-2 $, where $  \mathcal{S}(X,g)=\{ x\in \Sigma(X): g_0+dr^2+r^2k_l \text{ is not a Euclidean metric.} \} $ is a metric singularity set of $ (X,g) $.
			%\item  $ \mathcal{S}(X,g)\bigcap \Sigma^{n-2}(X)=\emptyset $, where $  \mathcal{S}(X,g)=\{ x\in \Sigma(X): g_0+dr^2+r^2k_l \text{ is not a Euclidean metric.} \} $ 
		\end{enumerate}
	\end{definition}
	
	In particular, by \cite[Theorem A]{BKMR2021} a compact $ n $-dimensional $ \lambda $-Einstein stratified space is a $ RCD(\lambda,n) $ metric measure space.
	
	Condition (2) required in the above definition is partially motivated by the codimension $ 4 $ theorem by Cheeger-Naber \cite{Cheeger-Naber2015} (i.e., the singular set in a noncollapsed limit space of Riemannian $ n $-manifolds $ (M_i,g_i) $ satisfying $ |\op{Ric}_{g_i}|\le n-1 $ and $ \op{Vol}(B_1(p_i)) >v>0 $ for all $ p_i\in M_i $ has Hausdorff dimension $ \le n-4 $), and the following Theorem \ref{thm-sphere-stratified} for Einstein stratified spaces proved by Honda-Mondello \cite{Honda-Mondello}.

	\begin{theorem}[Sphere theorem for Einstein stratified spaces {\cite[Theorem 3.3]{Honda-Mondello}}]\label{thm-sphere-stratified}
	Given integer $ n\ge 2 $, there exists $ \epsilon(n)>0 $ such that if a compact $ (n-1) $-Einstein stratified space $ (X,g) $ of dimension $ n $ satisfies
	$$ \frac{\mathcal{H}^n(X,g)}{\mathcal{H}^n (\mathbb{S}^n)}\ge 1-\epsilon(n), $$
	then $ (X,g) $ is isometric to a round sphere $ \mathbb{S}^n $.
	\end{theorem}
	
	The rigidity in Theorem \ref{thm-sphere-stratified} fails for those stratified spaces with iterated edge metric satisfying the condition (1) in Definition \ref{def-Einstein-stratified-space} in the absence of (2); see \cite[Remark 3.4]{Honda-Mondello} for a counter example.
	
	In the end of this paper, we point out that all the main theorems hold for singular Einstein spaces, such as Einstein stratified spaces and noncollapsed Ricci limit spaces of Einstein $ n $-manifolds.
	
	\begin{theorem}\label{Thm-stratified-rigidity}
	Theorems \ref{thm-rigidity-almostflat}, \ref{thm-gap}, \ref{thm-rigid-localrewindingvol} hold for compact $ \lambda $-Einstein stratified spaces $ (X,g) $ of dimension $ n $.
	\end{theorem}

	\begin{theorem}\label{thm-Ric-limit-rigidity}
	Theorems \ref{thm-rigidity-almostflat}, \ref{thm-gap}, \ref{thm-rigid-localrewindingvol} hold for noncollapsed Ricci limit spaces $ X $ of $ n $-manifolds $ (M,g) $ with bounded Ricci curvature $ |\op{Ric}_g|\le n-1 $, provided that the Riemannian metric $ h $ on the regular set of $ X $ satisfies $ \op{Ric}_h=\lambda h $.
	\end{theorem}
	
	Since the proof of Theorem \ref{thm-Ric-limit-rigidity} is similar to Theorem \ref{thm-sphere-stratified} and Theorem \ref{Thm-stratified-rigidity}, we only prove Theorem \ref{Thm-stratified-rigidity}.
	
	\begin{proof}[Proof of Theorem \ref{Thm-stratified-rigidity}]
	~
	
	It suffices to show that if a compact $ \lambda $-Einstein stratified space $ (X,g) $ satisfies the conditions in each of Theorems \ref{thm-rigidity-almostflat}, \ref{thm-gap}, \ref{thm-rigid-localrewindingvol}, then it is a smooth manifold without singular strata. Equivalently, any point $ x\in X $ is regular. 
	
	Observe that, if a stratified space $ (X,g) $ with an iterated edge metric satisfies $ \Sigma^{n-2}=\emptyset $, then the unit tangent sphere $ (S_x,h_x) $ at any $ x\in X $ is also $ \Sigma^{(n-1)-2}(S_x)=\emptyset $; see \cite[Lemma 3.1]{Honda-Mondello}. By calculating the Ricci curvature of the warped product metric, the unit tangent sphere $ (S_x,h_x) $ is a $ (n-1) $-Einstein stratified space. Since any point $ x\in(X,g) $ is regular if and only if $ (S_x,h_x) $ is isometric to a round sphere,  by Theorem \ref{thm-sphere-stratified} we only need to show that the unit tangent sphere $ (S_x,h_x) $ at $ x $ has almost maximal volume, i.e., 
	\begin{equation}\label{equ-almost-volume}
		\mathcal{H}^{n-1}(S_x)\ge (1-\epsilon(n-1))\mathcal{H}^{n-1} (\mathbb{S}^{n-1}).
	\end{equation}
	
	In the following, we verify \eqref{equ-almost-volume} under the conditions in Theorems \ref{thm-rigidity-almostflat}, \ref{thm-gap}, \ref{thm-rigid-localrewindingvol} one by one.
	
	In the case of Theorem \ref{thm-rigid-localrewindingvol}, because the universal cover $ \widetilde{B_\rho(x)} $ of $ B_\rho(x) $ satisfies $ \mathcal{H}^n(B_\rho(x^*))\ge (1-\epsilon)\mathcal{H}^n(B_\rho^H) $, we derive $ \mathcal{H}^n(B_s(x^*))\ge (1-\epsilon)\mathcal{H}^n(B_s^H) $, $ \forall s<\rho $ by Bishop-Gromov inequalities. 
	%The tangent cone  $ (C(S_{x^*}), dr^2+r^2h_{x^*},\mathcal{H}^n,o) $ is the pointed measured Gromov-Hausdorff limit space of $ (X,s^{-2}g,\mathcal{H}^n,x^*) $, as $ s\to 0 $. 
	By \eqref{def-iterated-metric} in the definition of iterated edge metric, the Hausdorff measure $ \mathcal{H}^n(B_1(x^*),s^{-2}g) $ of the unit ball in $ (X,s^{-2}g,x^*) $ converges to the Hausdorff measure $ \mathcal{H}^n(B_1(o)) $ of the unit ball in $ (C(S_{x^*}), dr^2+r^2h_{x^*},\mathcal{H}^n,o) $ as $ s\to 0 $.
	Thus, we have $ \mathcal{H}^n(B_1(o))\ge (1-\epsilon)\mathcal{H}^n(B_1^0) $, which implies
	$$ \mathcal{H}^{n-1}(S_{x^*})=n\mathcal{H}^n(B_1(o)) \ge n(1-\epsilon)\mathcal{H}^n(B_1^0) = (1-\epsilon)\mathcal{H}^{n-1}(\mathbb{S}^{n-1}).$$
	
	In the case of Theorem \ref{thm-gap}, $ (X,g,\mathcal{H}^n) $ is homeomorphic and $ \varkappa(\epsilon|n,D)  $-measured-Gromov-Hausdorff close to a hyperbolic $ n $-manifold $ \mathbb{H}^n/\Gamma $ by Chen-Xu \cite[Theorem 1.1]{CX2024}. So its universal cover $ \tilde{X} $ satisfies $ \mathcal{H}^n(B_r(\tilde{x}))\ge (1-\varkappa(\epsilon|n,D))\mathcal{H}^n(B_r^{-1}) $ for all $ x\in X $ and all $ r\in(0,1) $, where $ \tilde{x} $ is a preimage point of $ x $ in $ \tilde{X} $. By the same argument in the above case of Theorem \ref{thm-rigid-localrewindingvol}, the unit tangent sphere satisfies \eqref{equ-almost-volume}.
	
	In the case of Theorem  \ref{thm-rigidity-almostflat}, first note that it was proved by Zamora-Zhu \cite{Zamora-Zhu2024} that
	Theorem \ref{thm-almostflat} still holds for any noncollapsed $ RCD(-(n-1),n) $ metric measure spaces after replacing the diffeomorphism condition with homeomorphism in condition (1). 
	
	In fact, under the conditions of Theorem \ref{thm-almostflat} if a noncollapsed $ RCD(-(n-1),n) $ metric measure space $ (X,d,\mathcal{H}^n) $ satisfies (2), then by the argument of \cite[Theorem 4.1]{Zamora-Zhu2024} it also satisfies (3).
	In particular, after rescaling the metric to $ \epsilon^{-\frac{1}{2}} d $, the universal cover $ (\tilde{X},\tilde{x},\epsilon^{-\frac{1}{2}}d_{\tilde{X}},\mathcal{H}^n) $ is $ \varkappa(\epsilon|n) $-close to Euclidean space $ (\mathbb{R}^n,0,g_{\mathbb{R}^n},\mathcal{H}^n) $ in the pointed measured Gromov-Hausdorff distance, where $ \epsilon $ is the diameter of $ (X,d,\mathcal{H}^n) $, and $ d_{\tilde{X}} $ is the pull back metric on the universal cover of $ (X,d,\mathcal{H}^n) $. Hence, the volume of unit ball $ B_1(o) $ on the tangent cone $ C(S_x) $ at any $ x\in X $ is almost maximal.
	\end{proof}

	\section{Open Questions}
	The main theorems of this article lead to several open questions.
	
	Firstly, since $ \epsilon(n) $ in Corollary \ref{cor-almost-flat-Einstein-mfd} is apriorily smaller than that in Gromov's theorem on almost flat manifolds, it is natural to ask the following.
	\begin{problem}\label{problem-1}
	Is there any Einstein metric with negative scalar curvature on a nilmanifold or infra-nilmanifold of dimension $ n\ge 5 $?
	\end{problem}
	What is known about Problem \ref{problem-1} is the case that $n=2,3,4$ where Einstein structure must be flat; see \cite{Besse1987}.
	
	Secondly, let us consider the rescaling invariant $ \op{diam}^2\cdot \op{Ric} $ of an Einstein Riemannian $ n $-manifold. By the proof of Theorem \ref{thm-rigidity-almostflat}, we have the following corollary.
	\begin{corollary}\label{cor-diam-bounded}
	For any $ n\ge 2 $ and $ v>0 $, there is $ D(n,v)>0 $ such that for any Einstein $n$-manifold
	$(M,g)$ with $\op{Ric}_g=-(n-1) g$, if $\op{vol} B_1(\tilde x)\ge v>0$ at a point $ \tilde{x} $ in the Riemannian universal cover $ \tilde{M} $ of $ M $, then $ \op{diam}(M,g)\ge D(n,v) $.
	\end{corollary}
	
	By Gromov \cite[Section 4, Appendix E]{Gromov1983} and \cite{Gromov1982}, if $ (M,g) $ is homeomorphic to a closed aspherical manifold and satisfies $ \op{Ric}_g\ge -(n-1) $ and $ \op{diam}(M,g)\le D $, then $ \op{vol}(B_1(\tilde{x})) \ge c(n,D) $, for any $ \tilde{x}\in \tilde{M} $. Combining with Corollary \ref{cor-diam-bounded}, we have
	
	\begin{corollary}\label{cor-rescaling-invariant}
	For any compact Einstein and aspherical $ n $-manifold $ (M,g) $ with $\op{Ric}_g=\lambda g$, either $ \op{diam}(M,g)^2\cdot \lambda \le -D(n)^2<0 $ or $ (M,g) $ is flat.
	\end{corollary}
	
	Note that there are many compact Einstein $ n $-solvmanifolds with negative scalar curvature satisfying Corollary \ref{cor-rescaling-invariant}; see \cite{Alekseevskiui1975} \cite{Heber1998} \cite{Bohm-Lafuente} \cite{Lauret2010}.
	
	Lastly, it is worth mentioning the following problem raised by Thomas Schick.
	\begin{problem}\label{problem-2}
	Let $ D(n,v) $ be the infimum of the diameters $ \op{diam}(M,g) $ of all Einstein $n$-manifolds
	$(M,g)$ such that $\op{Ric}_g=-(n-1) g$ and $\op{vol} B_1(\tilde x)\ge v>0$ for some $ \tilde{x} \in \tilde{M} $. What kind of such Einstein manifold $ (M,g) $ realizing the extremal diameter $ D(n,v) $ can be?
	\end{problem}
	It is natural to conjecture that only solvmanifolds can realize such extremal diameter.

	%\bibliographystyle{plain}
	%\bibliography{document_0.7}

\end{document}